\documentclass{isamms}

\newtheorem{theor}{Theorem}
\newtheorem{pr}{Proposition}
\newtheorem{cl}{Corollary}

\theoremstyle{definition}

\newtheorem{ex}{Example}
\theoremstyle{remark}

\numberwithin{equation}{section}

\begin{document}
\runningtitle{On some classes of rings and their links}
\title{On some classes of rings and their links}

\author[1]{Hamideh Pourtaherian}
\cauthor
\author[2]{Isamiddin S.Rakhimov}

\address[1]{Department of Mathematics, FS, Universiti Putra Malaysia (Malaysia), \email{pourtaherian09@gmail.com}}

\address[2]{Institute for Mathematical Research (INSPEM), Department of Mathematics, FS, Universiti Putra Malaysia (Malaysia)
\addrbreak
Institute of Mathematics (Uzbekistan), \email{risamiddin@gmail.com}}

\authorheadline{Hamideh Pourtaherian and Isamiddin S.Rakhimov}

%

\begin{abstract}
The paper deals with Armendariz rings, their relationships with
some well known rings. Then we treat generalizations of Armendariz
rings, such as McCoy ring, abelian ring and their links. We also consider a skew version of some classes of rings, with
respect to a ring endomorphism $\alpha. $
\end{abstract}
\maketitle


\section{Introduction}
This paper investigates a class of rings called Armendariz rings,
which generalizes fields and integral domains. These rings are
associative with identity and they have been introduced by Rege and
Chhawchharia in \cite{rege}. A ring $R$ is
called \emph{Armendariz} if whenever the product of any two
polynomials in $R[x]$ is zero, then so is the product of any pair
of coefficients from the two polynomials.

In \cite{ep} E.P.Armendariz proved that if the product of two
polynomials, whose coefficients belong to a ring without nonzero
nilpotent elements, equals zero then all possible pair wise
products of coefficients of these polynomials equal zero.

Let $R$ be a ring and $\alpha :R\longrightarrow R$ be an
endomorphism. Then $\alpha $-derivation $\delta $ of $R$ is an
additive map such that $\delta (ab)=\delta (a)b+\alpha (a)\delta
(b),$ for all $a,b\in R.$ \emph{The Ore extension} $R[x;\alpha
,\delta ]$ of $R$ is the ring with the new multiplication $%
xr=\alpha (r)x+\delta (r)$ in the polynomial ring over $R,$
where $r\in R.$ If $\delta =0,$ we write $%
R[x;\alpha ]$ and it is said to be a skew polynomial ring (also
\emph{The Ore extension of endomorphism type}.)

Some properties of skew polynomial rings have been studied in
\cite{c1}, \cite{skew}, \cite{a}, \cite{?} and \cite{??}.
According to Krempa \cite{k1}, an endomorphism $\alpha $ of a ring
$R$ is called rigid, if for $r\in R$ the condition $r\alpha (r)=0$
implies $r=0$ . In \cite{rigid}, a ring $R$ has been called
$\alpha $-rigid if there exists a rigid endomorphism $\alpha $ of
$R.$ In the same paper it has been shown also that any rigid
endomorphism of a ring is a monomorphism and $\alpha$-rigid rings
are reduced.

 Hong et al. \cite{a}, introduced the concept of $\alpha$-Armendariz
 ring, which is a generalization of $\alpha$-rigid ring and Armendariz ring.
 A ring $R$ is called $\alpha$-Armendariz ring, if whenever the product of any two polynomials
 in $R[x; \alpha]$ is zero, then so is the product
 of any pair of coefficients from the two polynomials.

The organization of the paper is as follows. First, we consider
the relationship between Armendariz rings and some other classes
of rings (Section 1), then we treat generalizations of Armendariz
rings (Section 2). The skew version of some
classes of rings are considered in Section 3. Through the paper $\alpha$ stands for an
endomorphism of ring $R.$

\section{Armendariz rings and other rings}
In this section we explore relationships between several classes
of rings. Recall that a ring
$R$ is said to be \emph{von Neumann regular}, if $a \in aRa$ for any
element $a$ of $R. $ Every Boolean ring is von Neumann regular.
Reduced rings are Armendariz, but the converse does not hold.
Anderson and Camillo (see \cite{an}) proved that a von Neumann
regular ring is Armendariz, if and only if it is reduced.

\begin{pr} A commutative von Neumann regular ring is reduced.
\end{pr}
\begin{proof}
 Let $R$ be a commutative von Neumann regular ring and $a$ be an
element of $R.$ Suppose that $a^{2}=0.$ By the hypothesis, $a=a b
a=a^{2}b=0. $ Hence $R$ is reduced.
\end{proof}

By Kaplansky \cite{k}, a ring $R$ is called a right $p.p$-ring, if
the right annihilator $Ann_r(a)$ of each element $a$ of $R$ is generated by an
idempotent. A ring $R$ is called \emph{Baer}, if the right annihilator of
every nonempty subset of $R$ is generated by an idempotent.
Clearly Baer ring is right $p.p$-ring. Any Baer ring has nonzero
central nilpotent element, then a commutative Baer ring is
Armendariz.

Reduced rings can be included in the class of Armendariz rings and
the class of semicommutative rings. The last two are abelian.
It is natural to explore the relationships between them. A ring is
said to be \emph{semicommutative}, if it satisfies the following
condition:
$$\mbox{whenever elements}\ a,b \in R \ \mbox{satisfy}\ ab=0, \ \mbox{then}\
aRb=0 .$$
 Semicommutative rings are abelian, but the converse
does not hold, which has been showed by Kim and Lee in \cite{lee}.

Another class of rings is the class of Guassian rings, which has been treated by Anderson
and Camillo \cite{an}. The content $c(f)$ of a polynomial $f(x)\in R[x]$ is the ideal
of $R$ generated by the coefficients of $f(x).$ A commutative
ring $R$ with identity is \emph{Guassian}, if $c(fg)=c(f) c(g)$
for all $f(x), g(x) \in R[x]. $ The Guassian rings are Armendariz, but
the converse is not true. Any integral domain is Armendariz, but
it is not necessarily Guassian. A field is Guassian, thus it is
Armendariz.

Recall that, a ring $R$ is called \emph{symmetric}, if $abc=0$
implies $acb=0$ for $a, b$ and $c$ in $R. $ A ring $R$
\emph{reversible} provided $ab=0$ implies $ba=0$ for $a, b \in R.
$
Semicommutative ring is a generalization of reversible ring. A ring is said to be \emph{abelian} if any its idempotent is central. Further we make use the notation ``$\implies$'' to denote for one class of rings to be a subclass of another class.

\begin{theor} The following implications
hold true:
$$reduced \implies  symmetric \implies
reversible \implies  semi-commutative \implies  abelian.$$
\end{theor}
\begin{proof}
\begin{enumerate}\item ``$reduced \implies  symmetric:$'' Let $R$ be a reduced ring and $abc=0$ for
$a, b, c \in R. $ Then $c(abc)ab=0$ and $(cab)^{2}=0. $ Since $R$
is reduced, we get $(cab)=0. $ Hence $aba(cab)ac=(abac)^{2}=0$ and
$abac=0$ (by reducibility). Thus $bacb(abac)ba=(bacba)^{2}=0$ then
$bacba=0. $  Multiply the last from the right hand side by $c$ we
obtain $(bac)^{2}=0. $ By using the reducibility of $R$ we have
$bac=0. $
\item ``$symmetric \implies reversible:$'' Let $R$
be a symmetric ring and $ab=0$ for $a, b \in R. $ Since $R$ is a
ring with identity, we have $a\cdot b \cdot 1=0$ and $b \cdot a
\cdot 1=0. $ Therefore $R$ is reversible ring.
 \item ``$reversible \implies semicommutative:$'' Let $R$ be a reversible ring
and $ab=0. $  We claim that $aRb=0. $ By the hypothesis, $ba=0. $
Let $c$ be an arbitrary element of $R. $ Hence $c(ba)=0, $ and
$(cb)a=0. $ By using the reversibility of $R$ we have $acb=0.$
Thus $aRb=0. $ Therefore $R$ is semicommutative ring.
\item ``$semi-commutative \implies abelian:$'' Let $e$ be an
idempotent element of a semicommutative ring $R.$ Then
$e^{2}-e=0.$ Since $ea(e-1)=0$ for each element $a$ of $R$, we get
$ea=eae.$ Since $(1-e)$ is idempotent, then $(1-e)^{2}-1+e=0. $
Hence $(e-1)e=0. $ By using the semicommutativity of $R$ we
obtain $(e-1)ae=0$ for each $a \in R. $ Then $eae=ae. $ Thus $e$
is central. Therefore $R$ is abelian.
\end{enumerate}
\end{proof}

Here is an example of ring that is commutative, Boolean, von
Neumann regular, $p.p.-$ring, reduced and Armendariz, but is not
Baer (see \cite{c}).
 \begin{ex}(Dorroh extension) We refer the example of
\cite{c}. Let $S_{0}=\mathbb{Z}_{2}$, $S_{1}=\mathbb{Z}_{2}\ast
\mathbb{Z}_{2}$, $S_{3}=S_{2}\ast \mathbb{Z}_{2}$,..., $
S_{n}=S_{n-1}\ast \mathbb{Z}_{2},...,$ where the operation on
$S_{n}$ is defined as follows: for $(a,\bar{b}),(c,\bar{d})\in
S_{n}$ with $a,c\in S_{n-1}$
$$(a,\bar{b})+(c,\bar{d})=(a+c,\overline{b+d})$$ and $$
(a,\bar{b})(c,\bar{d})=(ac+bc+da,\overline{bd})$$ where
$n=1,2,\ldots.$ \\It is clear that there is the ring-monomorphism
$\ f:S_{n-1}\longrightarrow S_{n}$ defined by $ f(x)=(x,0)$. Now
construct the direct product $\ \prod\limits_{n=1}^{\infty }S_{n}$
with $S_{1}\subset S_{2}\subset ...$ and consider $R=\left\langle
\bigoplus\limits_{n=1}^{\infty }S_{n},1_{s}\right\rangle.$
Clearly, $R$ is a $\mathbb{Z}_{2}$ -subalgebra of
$\prod\limits_{n=1}^{\infty }S_{n},$ generated by $
\bigoplus\limits_{n=1}^{\infty }S_{n}$ and $1_{s},$  where
$S=\prod\limits_{n=1}^{\infty }S_{n}.$ Every $S_{n}$ is Boolean
and also von Neumann regular. Therefore $R$ is commutative von
Neumann regular ring, we get $R$ is Armendariz. On the other hand,
every $S_{n}$ is Boolean, then $R$ is Boolean. This implies that
$R$ is a $p.p.-$ring.
\end{ex}

\section{Generalizations of Amendariz rings}
 Abelian rings are generalization of Armendariz rings. This result due to Kim and Lee \cite{lee}. The following theorem specifies a subclass of the class of abelian rings which is Armendariz.
\begin{pr}
An abelian right $p.p.-$ring is Armendariz.
\end{pr}
\begin{proof}
Let $r$ be a nilpotent and $e$ be an idempotent elements of $R.$
Suppose that $r^{2}=0.$ Since $r\in Ann_{R}\left( r\right) =eR$,
there exists $r'\in R$ such that $r=er'$ and $er=e^{2}r'=er'.$
Hence $r=er=re=0. $ Which means that $R$ is reduced, therefore it is Armendariz.
\end{proof}

McCoy rings are another generalization of Armendariz rings. Recall
that a ring $R$ is a \emph{left McCoy}, if whenever $g(x)$  is a
right zero-divisor in $R[x]$ there exists a non-zero element $c
\in R$ such that $cg(x)=0. $ \emph{Right McCoy} ring is defined
dually. A ring is said to be \emph{McCoy} ring, if it is both left
and right McCoy. Armendariz rings are McCoy (see \cite{rege}), the
converse does not hold. Commutative rings are McCoy (Scott
\cite{s}), but there are examples of commutative non-Armendariz
rings.

Here is an example of a noncommutative McCoy ring that is not
Armendariz.
\begin{ex}\emph{}\\
Let $R$ be a reduced ring and $R_{n}=\left\{\left(
\begin{array}{ccccccc}
a & a_{12} & a_{13} & . & . & . & a_{1n} \\
0 & a & a_{23} & . & . & . & a_{2n} \\
0 & 0 & a & . & . & . & a_{3n} \\
. & . & . & . &  &  & . \\
. & . & . &  & . &  & . \\
. & . & . &  &  & . & . \\
0 & 0 & 0 & . & . & . & a
\end{array}%
\right) :a,a_{ij}\in R\right\}.$

Then $R_{n}$ is McCoy for any
$n\geq 1$ \cite{mc}, but it is not Armendariz for $n\geq 4$
\cite{lee}.
\end{ex}

All the previous results are included in Figure~1 below.
\begin{figure}[!hbtp]
\begin{center}
\includegraphics[width=4 in]{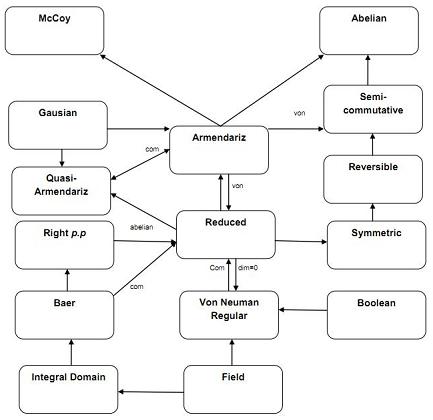}
\caption{Links between Armendariz and other rings.}
 \end{center}
 \label{Fig1}
 \end{figure}


\section{Skew version of rings }
In this section, we consider a skew version of some classes of
rings, with respect to a ring endomorphism $\alpha. $ When $\alpha
$ is the identity endomorphism, this coincides with the notion of
ring.

Kwak \cite{sym}, called an endomorphism $\alpha $ of a ring $R, $
\emph{right} (respectively, \emph{left}) \emph{symmetric} if whenever $abc=0$ implies
$ac\alpha (b)=0$ (respectively, $\alpha (b)ac=0)$ for $a,b,c \in R. $ A ring $R$
is called \emph{right} (respectively, \emph{left}) $\alpha $-\emph{symmetric} if there exists
a right (respectively, left) symmetric endomorphism $\alpha $ of $R. $ The ring
$R$ is \emph{$\alpha
$-symmetric} if it is right and left $\alpha $-symmetric.
Obviously, domains are $\alpha $-symmetric for any endomorphism
$\alpha.$

Ba\c{s}er et al. \cite{rev}, called a ring $R$ \emph{right} (respectively, \emph{left})
$\alpha$- \emph{reversible} if whenever $ab=0$ for $a, b \in R$ then $b
\alpha(a)=0$ (respectively, $\alpha (b)a=0$). The ring $R$ is called
\emph{$\alpha$- reversible} if it is both right and left $\alpha$-
reversible.

\begin{pr} An $\alpha$-symmetric ring is $\alpha$-reversible.
\end{pr}
\begin{proof}Let $R$ be an $\alpha$-symmetric ring. Suppose that
$ab=0$ for $a, b \in R. $ Obviously, $1 \cdot a \cdot b=0. $ Since
$R$ is right $\alpha$-symmetric, then $b\alpha(a)=0. $ Hence $R$
is right $\alpha$-reversible. It can be easily shown that $R$ is
left $\alpha$-reversible by the same way as above. Therefore
$R$ is $\alpha$-reversible.
\end{proof}

Baser et al. (see \cite{baser}), defined the notion of an
$\alpha$- semicommutative ring with the endomorphism $\alpha$ as a
generalization of $\alpha$-rigid ring and an extension of
semicommutative ring.

An endomorphism $\alpha$ of a ring $R$ is called
\emph{semicommutative}
 if $ab=0$ implies $a R \alpha(b)=0$ for $a,b \in R.$ A ring $R$ is called
\emph{$\alpha$-semicommutative} if there exists a semicommutative
endomorphism $\alpha$ of $R.$\\

\begin{pr}An $\alpha$-symmetric ring is $\alpha$-semicommutative.
\end{pr}
\begin{proof} Suppose that $ab=0$ for $a, b \in R. $ Let $c$ be an arbitrary
element of $R. $ Then $abc=0 $ and $ac \alpha(b)=0. $ Hence
$aR\alpha(b)=0. $ Therefore $R$ is $\alpha$-semicommutative ring.
\end{proof}

\begin{pr} A reduced $\alpha$-reversible ring is $\alpha$-semicommutative.
\end{pr}
\begin{proof} Suppose that $ab=0$ for $a, b \in R. $ Let $c$ be an arbitrary
element of a reduced $\alpha$-reversible ring $R. $ Then
$\alpha(b)a=0$ (by $\alpha$-reversibility) and $\alpha(b)ac=0. $
That is $ac\alpha (b)=0$ (by reducibility). Therefore $R$ is
$\alpha$-semicommutative.
\end{proof}

According to Hashemi and Moussavi \cite{c1}, a ring $R$ is
$\alpha$-compatible for each $a, b \in R, $ $a \alpha(b)=0$
 if and only if $ab=0$. Ben Yakoub and Louzari \cite{y}, called a ring $R$ satisfies the condition
($C_{\alpha}$) if whenever $a\alpha(b)=0$ with $a, b \in R, $ then
$ab=0. $ Clearly, $\alpha$- compatible ring satisfies the
condition ($C_{\alpha}$).

\begin{pr}An $\alpha$-reversible ring that satisfies the condition
$(C_{\alpha})$ is $\alpha$-semicommutative.
\end{pr}
\begin{proof}Suppose that $R$ is an $\alpha$-reversible ring with
($C_{\alpha}$) condition and $ab=0$ for $a, b \in R. $ Let $c$ be
an arbitrary element of $R. $ Hence $\alpha(b)a=0$ (by
$\alpha$-reversibility) and $\alpha(b)ac=0. $ Then $ac
\alpha^{2}(b)=0, $ due to $\alpha$-reversibility of $R$. Since $R$
satisfies the condition $(C_{\alpha}), $ we get $ac \alpha(b)=0. $
Therefore $R$ is $\alpha$-semicommutative.
\end{proof}

\begin{cl}An $\alpha$-reversible $\alpha$-compatible ring is $\alpha$-semicommutative.
\end{cl}
\begin{proof}
It is obvious.
\end{proof}

Ba\c{s}er et al \cite{baser}, proved that for
$\alpha$-semicommutative ring $R$, $\alpha(1)=1$ if and only if
$\alpha(e)=e, $ where $1$ is the identity and $e$ is the
idempotent element of $R. $

\begin{pr}  An $\alpha$-semicommutative ring $R$ with $\alpha(1)=1$
is abelian.
\end{pr}
\begin{proof} Let $e$ be an idempotent element of $R. $ Then $e(1-e)=0$
and $eR \alpha(1-e)=0. $ On the other hand, $(1-e)e=0$ and
$(1-e)R\alpha(e)=0. $\\ Hence $er(1-e)=(1-e)re=0$ for all $r \in
R, $ This implies that $er=re$ for all $r \in R. $ Therefore $R$
is abelian.
\end{proof}

The following example shows that the condition ``$\alpha(1)=1$''
can not be dropped.
\begin{ex} Let $R=\left\{\left(
\begin{array}{cc}
a & 0 \\
b & c
\end{array}
\right)| a,b,c \in \mathbb{Z} \right\},$ and $\alpha$ be an
endomorphism of $R$ defined by $\alpha\left(\left(
\begin{array}{cc}
a & 0 \\
b & c%
\end{array}%
\right)\right) =\left(
\begin{array}{cc}
0 & 0 \\
0& c
\end{array}%
\right), (\alpha(1)\neq 1). $\\
For $A=\left(
\begin{array}{cc}
a_{1} & 0 \\
b_{1}& c_{1}
\end{array}%
\right), $ and $B=\left(
\begin{array}{cc}
a_{2}  & 0 \\
b_{2}& c_{2}
\end{array}
\right)\in R $ if $AB=0, $ we obtain $c_{1}c_{2}=0. $ Let
$C=\left(
\begin{array}{cc}
a_{3}  & 0 \\
b_{3}& c_{3}
\end{array}
\right)$ be an arbitrary element of $R. $ Then $AC\alpha(B)=\left(
\begin{array}{cc}
0  & 0 \\
0& c_{1}c_{2}c_{3}
\end{array}
\right)=0. $ Hence $AR\alpha(B)=0. $ Therefore $R$ is
$\alpha$-semicommutative ring.\\
Two idempotents of $R$ (i.e., $\left(
\begin{array}{cc}
0  & 0 \\
b  & 1
\end{array}
\right)$ and $\left(
\begin{array}{cc}
1  & 0 \\
b & 0
\end{array}
\right)$) aren't central. Therefore $R$ is not abelian.
\end{ex}

In the next theorem we show the relationship between
semicommutative and $\alpha$-semicommutative rings.

\begin{theor} Let $R$ be an $\alpha$-compatible ring. Then the
following hold.
\begin{enumerate}
\item $R$ is symmetric if and only if $R$ is $\alpha$-symmetric
ring.
 \item  $R$ is reversible if and only if $R$ is $\alpha$-reversible ring.
\item $R$ is semicommutative if and only if $R$ is
 $\alpha$-semicommutative.
\end{enumerate}
\end{theor}
\begin{proof}
\begin{enumerate}
\item Let $R$ be a symmetric ring and  $abc=0,$ for $a, b, c \in R.
$ Then $acb=0$ (by symmetric property) and $ac \alpha(b)=0$ (by
$\alpha$-compatibility). Hence $R$ is right $\alpha$-symmetric.
Since $R$ is symmetric, then it is reversible and $\alpha(b)ac=0.
$ Thus $R$ is left $\alpha$-symmetric. Therefore
$R$ is $\alpha$-symmetric ring.\\
Conversely, let $R$ be an $\alpha$-symmetric ring and $abc=0$ for
$a, b, c \in R. $ Then $ac \alpha(b)=0$ and $acb=0$ (by
$\alpha$-compatibility). Therefore $R$ is symmetric ring. \item
Let $R$ be a reversible ring and ab=0, for $a, b \in R. $ Then
$ba=0. $ Hence $b\alpha(a)=0$(by $\alpha$-compatibility).
Therefore $R$ is right $\alpha$-reversible.\\
On the other hand, $ab=0$ we have $a \alpha(b)=0. $ Hence
$\alpha(b)a=0$(by reversibility). Thus $R$ is left
$\alpha$-reversible. Therefore $R$ is $\alpha$-reversible.\\
Conversely, let  $ab=0$ for $a, b \in R. $ Then $b\alpha(a)=0$ (by
right $\alpha$-reversibility) and $ba=0$(by
$\alpha$-compatibility). Therefore $R$ is reversible. \item Let
$R$ be a semicommutative ring and $ab=0$ for $a, b \in R. $ Hence
$aRb=0$. Since $R$ is $\alpha$-compatible, it implies that $aR
\alpha(b)=0. $ Therefore $R$ is $\alpha$-semicommutative ring. The
``only if'' part is obvious.
 \end{enumerate}
\end{proof}

\begin{pr} Let $R$ be an
$\alpha$-semicommutative ring with ($C_{\alpha}$) condition then
$R$ is semicommutative.
\end{pr}
\begin{proof}Let $R$ be an $\alpha$-semicommutative ring and
$ab=0$ for $a, b \in R. $ Then $aR \alpha(b)=0. $ Since $R$
satisfies the condition ($C_{\alpha}$), we get $aRb=0. $ Therefore
$R$ is semicommutative ring.
\end{proof}

All the previous results are summarized in Figure~2.
\begin{figure}[!hbtp]
\begin{center}
\includegraphics[width=5 in]{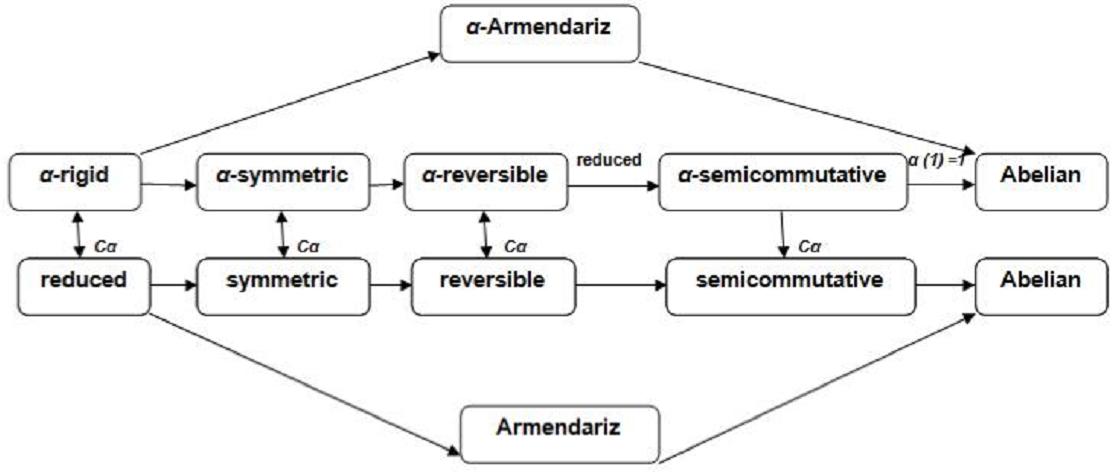}
\caption{\label{Fig2_} Links between skew
version of rings.}
 \end{center}
 \end{figure}


\newpage


\begin{thebibliography}{9}

\bibitem{an} D.D. Anderson and V. Camillo, Armendariz rings and Gaussian
rings, \emph{Comm. Algebra}, 26(7) (1998), 2265-2272.

\bibitem{ep} E.P. Armendariz, A note on extesions of Baer and $p.p.-$ring,
\emph{J. Aust. Math. Soc.}, 18 (1974), 470-473.

\bibitem{baser} M. Ba\c{s}er , A. Harmanci, and T.K.Kwak, Generalized semicommutative rings and their extensions,
\emph{Bull. Korean Math. Soc.}, 45(2) (2008), 285-297.

 \bibitem{rev} M. Ba\c{s}er , C.Y. Hong, and T.K. Kwak, On extended
 reversible rings, \emph{Algebra Colloq.}, 16(1) (2009), 37-48.

\bibitem{y} L. Ben Yakoub and M. Louzari, Ore extensions of extended symmetric
 and reversible rings, \emph {International Journal of Algebra},  3(9) (2009), 423-433.

\bibitem{c1}E. Hashemi and A. Moussavi, Polynomial extensions of
quasi-Baer rings, \emph {Acta. Math. Hungar.,} 107(3) (2005),
207-224.

\bibitem{skew} C.Y. Hong, N.K. Kim and T.K. Kwak, On skew Armendariz rings,%
\emph{\ Comm. Algebra}, 31(1) (2003), 103-122.

\bibitem{rigid} C.Y. Hong, N.K. Kim and T.K. Kwak, Ore extensions of Baer and
$p.p.$-rings, \emph{J. Pure Appl. Algebra}, 151 (2000), 215-226.

\bibitem{a} C.Y. Hong, T.K. Kwak and S.T. Rizvi, Extensions of generalized
Armendariz rings, \emph{Algebra Colloq.}, 13(2) (2006), 253-266.

\bibitem{k} I. Kaplansky, Rings of operations, Mathematics Lecture Notes
series, Benjamin, New York (1965).

\bibitem{lee} N.K. Kim and Y. Lee, Armendariz rings and reduced rings, \emph{J. Algebra}, 223(2) (2000), 477-488.

\bibitem{k1} J. Krempa, Some examples of reduced rings, \emph{Algebra Colloq.},
3(4) (1996), 289-300.

\bibitem{sym}T.K. Kwak, Extensions of extended symmetric rings, \emph{Bull. Korean Math. Soc.}, 44(4)
(2007), 777-788.

\bibitem{mc} Z. Lei, J. Chen and Z. Ying, A Question on McCoy rings, \emph{Bull.
Aust. Math. Soc.}, 76 (2007), 137-141.

\bibitem{c} Y. Lee, N.K. Kim  and C.Y. Hong, Counterexample on Baer rings,\emph{
Comm. Algebra}, 25(2) (1997), 497-507.

\bibitem{h} H. Pourtaherian and I.S. Rakhimov, On Armendariz Ring and its
generalizations,  \emph{JP Journal of Algebra, Number Theory and
Applications}, 15(2) (2009), 101 - 111.

\bibitem{?} H. Pourtaherian and I.S. Rakhimov, On extended quasi-Armendariz
rings, arXiv:1205.5845 v1.[math.RA], (2012).

\bibitem{??} H. Pourtaherian and I.S. Rakhimov, On skew version of reversible rings,
\emph{International Journal of Pure and Applied Mathematics},
73(3) (2011), 267-280.

\bibitem{rege} M.B. Rege, S. Chhawchharia, Armendariz rings, \emph{Proc.
Japan Acad. Ser. A Math. Sci.}, 73 (1997), 14-17.

\bibitem{s} W.R. Scott, Divisors of zero in polynomial rings,
\emph{Amer. Math. Monthly}, 61(5) (1954), 336.

\end{thebibliography}
\end{document}